\documentclass{amsart}
\usepackage{amsfonts,amscd}

\usepackage[arrow,matrix,curve,cmtip,ps]{xy}
\newtheorem{theorem}{Theorem}[section]

\newtheorem{corollary}[theorem]{Corollary}
\newtheorem{proposition}[theorem]{Proposition}
\theoremstyle{remark}

\theoremstyle{definition}
\newtheorem{definition}[theorem]{Definition}

\numberwithin{equation}{section} \makeatother

\begin{document}

\title[On a generalization of $W^*$-modules]{On a
generalization of $W^*$-modules}

\author{David P Blecher}

\address{Department of Mathematics, University of Houston, Houston, TX
77204-3008}

\author{Jon E Kraus}
\address{Department of Mathematics, State University of
New York at Buffalo, Buffalo, New York 14260-2900} \email[David P.
Blecher]{dblecher@math.uh.edu} \email[Jon E.
Kraus]{mthjek@acsu.buffalo.edu}

\date{\today}




\begin{abstract}  In
a recent paper of the first author and Kashyap,
a new class of modules over dual operator algebras is
introduced.  These generalize the $W^*$-modules
(that is, Hilbert $C^*$-modules over a von Neumann algebra
which satisfy an analogue of the Riesz representation
theorem for Hilbert spaces), which in turn generalize the
theory of Hilbert spaces.
In the present paper we give several
new results about these modules.
\end{abstract}

\maketitle

\section{Introduction and Notation}

We begin with two definitions of great importance in $C^*$-algebra theory,
 which may be found in more detail in \cite{Pas} for example.
A (Hilbert) $C^*$-module is a right module over a $C^*$-algebra $A$
 with an $A$-valued inner product satisfying a variant of the
usual axioms for a Hilbert space inner product.
A {\em $W^*$-module} is a $C^*$-module over a von Neumann algebra
which satisfies an analogue of the Riesz representation
theorem for Hilbert spaces.  Such spaces are far reaching and profound
generalizations of Hilbert space.
An earlier work \cite{Bsd} attempted to treat
$W^*$-modules in the framework of dual Banach modules where this
is possible, and where not possible then using the `operator space' variant of
dual Banach modules.  Recently, the first author and
Kashyap \cite{BK2} generalized the notion, and a large part
 of the theory, of $W^*$-modules to
the setting of `dual operator algebras' more general than von Neumann
algebras.  The modules introduced there are called {\em $w^*$-rigged modules}.
The present paper continues this work.
In Section 2, which is pedagogical in nature in keeping
with the nature of this volume, we motivate
the definition of $w^*$-rigged modules by
sketching a proof that $w^*$-rigged modules over
von Neumann
algebras are precisely the $W^*$-modules.  Indeed we give here
a simplification of the
main part of the proof from \cite{BK2}.  In Section 3,
we prove that tensoring with a $w^*$-rigged module
is bifunctorial in a certain sense.
As an application, we complete the functorial
characterization of Morita equivalence,
 sometimes called `Morita's fundamental theorem',  in our setting
of dual operator algebras.  One direction of this
characterization  is the main theorem
in \cite{UK1}, the other direction, from \cite{BK1}, was incomplete.
In Section 4 we initiate the study of an interesting class of $w^*$-rigged modules,
and prove an analogue of Paschke's powerful
characterization of $W^*$-modules as complemented submodules of `free' modules
\cite{Pas}.
  We also give a short proof of a variant of the main theorem from
\cite{EP}, the stable isomorphism theorem.
Finally, Section 5 characterizes the structure of surjective weak* continuous
linear complete
isometries between  bimodules, generalizing the well known `noncommutative
Banach-Stone' theorem for complete
isometries between operator algebras.    At first sight, it is not clear what
a structure theorem for linear isomorphisms $T$
between bimodules might look like.
Some thought and experience reveals that the theorem
one wants is precisely, or may be summarized succinctly by saying,
 that $T$ is the `restriction to the $1$-$2$ corner' of
a surjective (completely
isometric in our case)
homomorphism between the `linking algebras', which maps each of the 4 corners
to the matching corner.      We recall
that the linking algebra of a bimodule $Y$ is an (operator) algebra which consists
of $2 \times 2$ matrices whose four entries (`corners') are $Y$, its `dual module'
$X = \tilde{Y}$,
and the two algebras acting on the left and the right (see \cite[Section 4]{BK1} and 3.2
in \cite{BK2} for more details in our setting).
These linking algebras are a fundamental tool, being an  operator algebra
whose product encapsulates all the `data' of the module.  This `lifting of
$T$ to the linking algebra' is in the spirit of Solel's theorem
from  \cite{Sol} concerning isometries between $C^*$-modules.

We will assume from Section 3 onwards that the
reader has looked at the papers \cite{BK1,BK2}, in which also
further background and information may be found.   We also assume that the
reader is familiar with a few basic definitions from {\em operator space
theory}, as may be found in e.g.\ \cite{BLM,Pn}.
 In particular, we assume
that the reader knows what a dual operator space is,
and is familiar with basic Banach space and operator
space duality principles (as may be found for example within
  \cite[Section 1.4, 1.6,
Appendix A.2]{BLM}). We will often abbreviate `weak*' to `$w^*$'.
Unless indicated otherwise, throughout the paper $M$ denotes a {\em
dual operator algebra}, by which we mean a
weak* closed subalgebra of $B(H)$, the bounded operators on a
Hilbert space $H$.
We take all dual operator algebras $M$ to be
{\em unital}, that is we assume they possess an identity of norm
$1$.   Dual operator algebras may also be characterized abstractly
(see e.g.\ \cite[Section 2.7]{BLM}).
For the purposes of this paper,
a  (right) {\em dual operator module} over $M$ is an
(operator space and right $M$-module completely isometrically,
weak* homeomorphically, and $M$-module isomorphic to a)  weak* closed subspace
$Y \subset B(K,H)$, for Hilbert spaces $H, K$,
with $Y \pi(M) \subset Y$, where $\pi : M \to B(H)$ is
a weak* continuous completely contractive unital homomorphism.
Similarly for left modules.
Dual operator modules may also be characterized abstractly
(see e.g.\ \cite[Section 3.8]{BLM}).   For
right dual operator $M$-modules, the `space of morphisms' for
us will be $w^*CB(Y,Z)_M$, the weak* continuous completely bounded
right $M$-module maps.  We write $\tilde{Y} = w^*CB(Y,M)_M$, this is a left $M$-module, and plays the
role in our theory of the `module dual' of $Y$.

If $n \in \Bbb{N}$ and $M$ is a dual operator algebra,
then we write $C_n(M)$ for the first column of the space $M_n(M)$
of $n \times n$ matrices with entries in $M$.  This
is a right $M$-module, indeed is a dual operator $M$-module.  As one
 expects,
$\tilde{Y}$ is the `row space' $R_n(M)$ in this case.
Similarly if $n$ is replaced by an arbitrary cardinal  $I$:
 $C^w_I(M)$ may be viewed as one `column'
of the `infinite matrix
algebra' $\Bbb{M}_I(M) = M
 \bar{\otimes} B(\ell^2_I)$ (see \cite[2.7.5 (5)]{BLM}).  This is also
a dual operator $M$-module (e.g.\ see 2.7.5 (5) and p.\ 140 in \cite{BLM}).
These modules are the `basic building blocks' of $w^*$-rigged modules.

\section{$W^*$-rigged modules}

Although $C^*$-modules were generalized to the setting of nonselfadjoint operator algebras
in the 1990s, for at least a decade after that it was not clear how generalize
$W^*$-modules and their theory.  In \cite{BK2} we found the correct generalization, namely the
$w^*$-rigged modules.   There are now several equivalent definitions
of these objects (see \cite[Section 4]{BK2}),
of which the following is the most elementary:

\begin{definition} \label{wrig} Suppose that $Y$ is a dual operator
space and a right module over a dual operator algebra $M$. Suppose
that there exists a net of positive integers $(n(\alpha))$, and
$w^*$-continuous completely contractive $M$-module maps
$\phi_{\alpha} : Y \to C_{n(\alpha)}(M)$ and $\psi_{\alpha} :
C_{n(\alpha)}(M) \to Y$, with $\psi_{\alpha}( \phi_{\alpha}(y))  \to
y$ in the $w^*$-topology on $Y$, for all $y \in Y$.   Then we say
that $Y$ is a {\em right $w^*$-rigged module} over $M$.
\end{definition}

Our intention in this section is pedagogical.
We will give some motivation for this definition, and
use it to introduce
some ideas in the theory.  Suppose that $M$ is a von Neumann algebra, acting on
a Hilbert space $H$.  Henceforth in this section, let  $Y$ be a right Banach $M$-module satisfying the
Banach module variant of Definition \ref{wrig}, replacing `operator space' by `Banach space'
and `completely contractive' by `contractive'.   That every right $W^*$-module over $M$
is of this form follows exactly as in the Hilbert space case (where $M = \Bbb{C}$),
since there always exists
an `orthonormal basis' $(e_i)_{i \in I}$ (see \cite{Pas}).  In this case the net is indexed by the
finite subsets of $I$, and it is an easy exercise to write down the
maps $\phi_{\alpha}, \psi_{\alpha}$ in terms of the $e_i$.  We wish to show the
nontrivial converse, that any such Banach $M$-module $Y$ is
a $W^*$-module.  The bulk of this amounts to showing
that the weak$^{*}$ limit $w^*${\rm lim}$_{\alpha} \,
\phi_{\alpha}(y)^* \phi_{\alpha}(z)$ exists in $M$ for $y$, $z$
$\in$ $Y$: this expression then defines the $W^*$-module inner product.
However the existence of this weak$^{*}$ limit seems to be surprisingly difficult.
We will sketch a proof, giving full details of a new
proof of the main part of the argument.

It is a pleasant exercise for the reader (see
the first lemma in \cite{BK2} for a solution),
to check the case
that $M = \Bbb{C}$, that in this case $Y$ is a Hilbert space
with inner product lim$_{\alpha} \,
\langle \phi_{\alpha}(z),
\phi_{\alpha}(y) \rangle$, and this
will be used later in the proof. The next thing to note is that $\Vert y \Vert = \sup_\alpha \,
\Vert \phi_{\alpha}(y) \Vert$.  Indeed, if $\sup_\alpha \,
\Vert \phi_{\alpha}(y) \Vert \leq t < \Vert y \Vert$, then
$\Vert \psi_{\alpha}( \phi_{\alpha}(y)) \Vert \leq t$ for all
$\alpha$, and we obtain the contradiction $\Vert y \Vert \leq t < \Vert y \Vert$
from Alaoglu's theorem.  We remark in passing that a similar argument shows that
for any operator space $Y$ satisfying Definition \ref{wrig}, we have
\begin{equation} \label{yfor}  \Vert [y_{ij}] \Vert_{M_n(Y)}
= \sup_\alpha \; \Vert [\phi_{\alpha}(y_{ij})] \Vert , \qquad
[y_{ij}] \in M_n(Y).
\end{equation}
Note that equation (\ref{yfor}) implies that such a $Y$ is a dual operator module ($Y$ is
identified with a submodule
of a direct sum of the operator modules $C_{n(\alpha)}(M)$).
In our (Banach
module) case, we use (\ref{yfor}) as a {\em definition} of matrix norms.
This makes $Y$ an operator space, a dual operator module, and
it is easy to chck that $Y$ is now a $w^*$-rigged module.

We now mention that
the kind of tensor product that appears in our theory is called the {\em module $\sigma$-Haagerup
tensor product}  $Y \otimes^{\sigma h}_{M} Z$ (see \cite{EP} and  \cite[Section 2]{BK1}).
We will not take time to properly introduce this here, suffice it to say that this
tensor product is a dual operator space which is defined to have
 the universal property that it linearizes completely
bounded separately weak* continuous
bilinear maps satisfying $T(ym,z) = T(y,mz)$ for $y \in Y,z \in Z, m \in M$.
If $Y$ is a $w^*$-rigged $M$-module, and if $H$ is the Hilbert space
that $M$ acts on, then by tensoring on the right with
the identity map on $H$,
one can show that $K = Y \otimes^{\sigma h}_{M} H$ is $w^*$-rigged over $\Bbb{C}$.
By the exercise for the reader above, $K$
is a
Hilbert space with inner product
$$\langle y \otimes \zeta , z \otimes \eta \rangle =
\lim_\alpha \, \langle (\phi_{\alpha} \otimes 1)(y \otimes \zeta) ,
(\phi_{\alpha} \otimes 1)(z \otimes \eta) \rangle = \lim_\alpha \,
\langle \phi_{\alpha} (z)^* \phi_{\alpha}(y) \zeta , \eta \rangle,$$
for $y, z \in Y$ and $\zeta, \eta \in H$.   The computation above
also uses the simple fact that $C_n(M) \otimes^{\sigma h}_{M} H = H^{(n)}$
unitarily via the obvious map.   Thus
$\lim_\alpha \, \phi_{\alpha} (z)^* \phi_{\alpha}(y)$ exists in the
weak* topology of $M$ as desired.  Define $\langle z , y \rangle$ to be this
weak* limit in $M$.  It needs to be checked that
this matches the original norm.  This again uses the  fact that
this holds in the case $M= \Bbb{C}$
(the exercise for the reader above), and the just mentioned
`simple fact', as follows:  If $\zeta \in {\rm Ball}(H)$
then
$$
\Vert \phi_{\alpha}(y) \zeta \Vert^2 = \Vert (\phi_{\alpha} \otimes
1)(y \otimes \zeta) \Vert^2 \leq \Vert y \otimes \zeta \Vert^2 =
\langle \langle y , y \rangle \zeta , \zeta \rangle \leq \Vert y
\Vert^2.$$ Taking a supremum over such $\zeta$ and $\alpha$, and
using (\ref{yfor}) we obtain $$\Vert \langle y , y \rangle \Vert =
\sup_\alpha \, \Vert \phi_\alpha(y) \Vert^2 = \Vert y \Vert^2 ,
\qquad y \in Y .$$   Now it is clear that  $Y$ with its original
norm is a $C^*$-module over $M$.
That $Y$ is a $W^*$-module will now be clear to experts; but in any case
it is an easy fact that $W^*$-modules are the $C^*$-modules
whose inner product is separately weak* continuous (see e.g.\ Lemma 8.5.4 in \cite{BLM}).
The latter is clear in our case
using a basic fact about
$Y \otimes^{\sigma h}_{M} Z$: the map $\otimes$ is
separately weak* continuous (see \cite{EP} and \cite[Section
2]{BK1}).  For example, if $y_t  \to y$ weak* in $Y$ then $y_t \otimes \zeta
\to y \otimes \zeta$ weakly in $K$, hence
$$\langle \langle z , y_t \rangle \zeta , \eta \rangle =
\langle y_t \otimes \zeta , z \otimes \eta \rangle \to \langle y
\otimes \zeta , z \otimes \eta \rangle = \langle \langle z ,
y\rangle \zeta , \eta \rangle . $$

\medskip

{\bf Remark.}   It is tempting to
try to simplify the proof further by using
ultrapowers.  We thank Marius Junge for discussions on this;
it seems that such a proof, while very interesting, may be more complicated.

\section{Bifunctoriality of the tensor product, and an application}

The reader is directed to \cite{BK2} for the basic theory of
$w^*$-rigged modules.  Turning to new results, we first prove the important but nontrivial
fact that the
module $\sigma$-Haagerup
tensor product introduced in the last section is
`bifunctorial' in the following sense:

\begin{theorem} \label{FTM}  Suppose that $Y$ is a right $w^*$-rigged $M$-module,
and that $Z, W$ are left dual operator $M$-modules. If $(S_t)$ is
a net in $w^*CB_M(Z,W)$ with weak* limit $S \in w^*CB_M(Z,W)$,
then $I_Y \otimes S_t \to I_Y
\otimes S$ weak* in $CB(Y \otimes^{\sigma h}_M Z, Y
\otimes^{\sigma h}_M W)$.  Similarly, if $T_t \to T$ weak* in $w^*CB(Y)_M$
then $T_t \otimes I_Z \to T \otimes I_Z$ weak*.
\end{theorem}

\begin{proof}  The key new idea is to employ an isomorphism
found in  Theorem 3.5 of
\cite{BK2}: if  $\tilde{Y} = w^*CB(Y,M)_M$, then
the following map is an isometric
weak* homeomorphism: $\theta_Z : Y \otimes^{\sigma h}_M Z
\to w^*CB_M(\tilde{Y},Z)$,
 taking $y \otimes z$, for $y \in Y, z \in Z$,
 to the operator mapping
an $x \in \tilde{Y}$ to $(x,y) z \in Z$.
If $u = y \otimes z$, for $y, z$ as above, and $x \in \tilde{Y}$, then
$$(S \circ \theta_{Z})(u)(x) = S((x,y) z) = (x,y) Sz = \theta_{W} ((I
\otimes S)(u)) .$$   Thus $S \circ \theta_{Z}(u) = \theta_{W} ((I
\otimes S)(u))$ for $u = y \otimes z$.
By $w^*$-continuity and the density of
elementary tensors, it follows that $S \circ \theta_{Z}(u) = \theta_{W} ((I
\otimes S)(u))$, for all $u \in Y \otimes^{\sigma h}_M Z$ and $S
\in w^*CB_M(Z,W)$. If $S_t \to S$ weak* in $w^*CB_M(Z,W)$,
then $S_t \circ \theta_{Z}(u) \to S  \circ \theta_{Z}(u)$ weak*,
and hence $\theta_{W}((I \otimes S_t)(u))  \to \theta_{W}((I
\otimes S)(u))$ weak*.   Thus $(I \otimes S_t)(u) \to  (I \otimes S)(u)$ weak*.

 That $T_t \otimes I_Z \to T \otimes I_Z$ weak* is much shorter, following from
facts about operator space multipliers, as in the proof of Theorem 2.4 in
\cite{BK2}.
  \end{proof}

{\bf Remark.}  The variant of the last statement of the last theorem
for a net $(T_t)$ in $w^*CB(Y,Y')$, where $Y'$ is a second
$w^*$-rigged $M$-module, is also valid.  This may be seen via the
trick of viewing  $(T_t)$ in $w^*CB_M(Y \oplus^c Y', Y \oplus^c Y')$.

\medskip

In \cite{UK1}, Kashyap proved
the `difficult direction' of the analogue of one of Morita's famous
theorems:
 dual operator algebras are weak* Morita equivalent iff they are
{\em left dual Morita equivalent}
 in the sense of \cite[Definition 4.1]{UK1}.    Some of the `easy direction' of the
theorem was observed in \cite{BK1}.  However one aspect of this,
namely the weak* continuity of the
functors implementing the categorial equivalence, stumped us until we were able in
the present work to prove Theorem \ref{FTM}.

\begin{corollary} \label{ed}  If $M$ and $N$ are
weak* Morita equivalent in the sense of \cite{BK1}, then their
categories of dual operator modules are equivalent via functors that
are weak* continuous on morphism spaces.
That is, they are
left dual Morita equivalent  in the sense of {\rm \cite[Definition
4.1]{UK1}}.
\end{corollary}

\begin{proof}
Let $Y$ be the equivalence $N$-$M$-bimodule, with dual bimodule $X$.  Then
$Y$ is a $w^*$-rigged
$M$-module, and so the functor $Y \otimes^{\sigma h}_M \; \-- \;$ is weak*
continuous on spaces of morphisms, by Theorem \ref{FTM}.  Similarly,
$G = X \otimes^{\sigma h}_N  \;  \-- \;$ is weak* continuous.
That is, the functors implementing the categorial equivalence
are weak* continuous, which was the missing detail from our paper
\cite{BK1}.
\end{proof}

\section{A variant of Paschke's `free module' characterization}

In the last section of \cite{BK2}, we gave several examples of
$w^*$-rigged modules, including the following two.   First, if $P$ is a
weak*-continuous completely contractive idempotent $M$-module map on
$C^w_I(M)$, for a cardinal/set $I$, then Ran$(P)$ is a $w^*$-rigged
module.  Second, suppose that
 $Z$ is any
WTRO (that is, a weak* closed subspace of $B(K,H)$ satisfying
$Z Z^* Z \subset Z$), and suppose that $Z^* Z$ is contained in a dual operator
algebra $M$. Then $Y = \overline{ZM}^{w^*}$ is a $w^*$-rigged
$M$-module, and such modules $Y$
may be viewed as a one-sided generalization of the
bimodules studied in \cite{El,EP}. The following is the analogue of
a famous theorem due to Paschke (see \cite{Pas} or \cite[Corollary
8.5.25]{BLM}).  We identify two $w^*$-rigged modules {\em as dual
operator $M$-modules,}
if there is a surjective weak* homeomorphic completely isometric $M$-module
map between them.

\begin{theorem}  As dual operator $M$-modules, the above two classes  of $w^*$-rigged modules over $M$
coincide, and coincide with the class of $w^*$-rigged module direct
sums $\oplus^{wc}_{i  \in I} \; p_i M$, for sets of
projections $\{ p_i : i \in I \}$  in $M$ (this sum
coincides with the weak* closure in $C^w_I(M)$ of the
algebraic direct sum  $\oplus_{i  \in I} \; p_i M$).  \end{theorem}

\begin{proof}  Suppose that we have a projection $P \in w^*CB(C^w_I(M))_M \cong
\Bbb{M}_I(M)$ (see Corollary 3.6 in \cite{BK2}).   If $P = [a_{ij}]$ then
$[a_{ij}] = [a_{ji}^*]$, and so $P \in \Bbb{M}_I(N)$ where $N = M \cap
M^*$. Let $Z = P C^w_I(N)$, a WTRO.  Then $Z M \subset P C^w_I(M)$.
On the other hand,  if $\{ e_i \}$ is the usual `basis' for
$C^w_I(N)$ then $P e_i \in Z$, and so $P e_i m  \in ZM$ for all $m
\in M$. Hence $P C^w_I(M) \subset \overline{ZM}^{w^*}$, and so $P
C^w_I(M) = \overline{ZM}^{w^*}$.

Conversely,  suppose that $Y = \overline{ZM}^{w^*}$ as above.  Set
$R = \overline{Z^* Z}^{w^*}$. By the theorem of Paschke that we
are modifying \cite{Pas},
there exist mutually orthogonal partial isometries $(z_i)_{i \in I}
\subset Z$ with $\sum_i \, z_i z_i^* z = z$ for all $z \in Z$, and
if $P = [z_i^* z_j]$ then  $Z \cong P C^w_I(R)$.
The map $\theta : Y \to C^w_I(M) : x \mapsto [z_i^* x]$ is a weak*
continuous complete isometry with left inverse $[m_i] \mapsto \sum_i
\, z_i m_i$. It is easy to see that $\theta(Y) \subset P C^w_I(M)$.
Also, $\theta(z_j) = P e_i$ for $e_i$ as above, and so as at the end
of the last paragraph we have $P C^w_I(M) \subset \theta(Y)$.  Hence
$Y \cong P C^w_I(M)$ via $\theta$.

Finally, suppose that $p_i = z_i^* z_i$.  We
may view $\oplus^{wc}_i \, p_i M$ as a submodule of $C^w_I(M)$.  Clearly,
$\oplus^{wc}_i \, p_i M \subset P C^w_I(M)$, and the reverse
inclusion follows as in the first paragraph,  since $P e_i \in
\oplus^{wc}_i \, p_i M$.
\end{proof}

The modules considered in the last theorem
 form a very interesting subclass of the
$w^*$-rigged modules, and we propose a study of this subclass.
We shall call them {\em projectively $w^*$-rigged modules}.

\begin{proposition} \label{new}  Let $Y$ be a
 $w^*$-rigged right module over $M$, and let $X = \tilde{Y}$.
Then $Y$  is projectively $w^*$-rigged
iff there exists a pair $x = (x_i) \in {\rm Ball}(C^w_I(X)), y = (y_i) \in {\rm Ball}(R^w_I(Y))$, for a
cardinal $I$, such that $x_i(y_j) = \delta_{ij} \, p_j$ for an orthogonal projection $p_j \in M$, for
all $i,j \in I$, and $\sum_{i \in I} \, y_i \otimes x_i = 1$ weak* in $w^*CB(Y)_M$.
\end{proposition}

\begin{proof}  If $Y$ is projectively $w^*$-rigged, then we set $y_i = z_i,  x_i  = z_i^*$ in the
notation above.  Conversely, if such $x, y$ exist, then a slight variant
of the second paragraph of the proof of the last theorem
shows that $Y$ may be identified with $P C^w_I(M)$, where $P = [\delta_{ij} \, p_j]$,
   as dual operator $M$-modules.
\end{proof}

The following is a variant of the stable isomorphism theorem from \cite{EP}.  The
notation $\Bbb{M}_I(M)$ is defined in the introduction.

\begin{theorem}  Suppose that $Y$ is a weak* Morita equivalence $M$-$N$-bimodule,
 over dual operator algebras
$M$ and $N$, and suppose that $Y$ is both left and right projectively
$w^*$-rigged.   Then $M$ and $N$ are stably isomorphic (that is,
$\Bbb{M}_I(M) \cong \Bbb{M}_I(N)$ completely isometrically and weak* homeomorphically,
for some cardinal $I$).
  \end{theorem}

\begin{proof}  We will be brief, since this is well-trodden ground
(see Theorem 8.5.28 and Theorem 8.5.31 in \cite{BLM}).
 Since $Y$ is projectively $w^*$-rigged over $N$, we have $Y \cong P C^w_I(N)$ for a cardinal $I$.
If $Z = (I-P) C^w_I(N)$ then $Y \oplus^c Z \cong C^w_I(N)$ as dual operator $N$-modules.
Since $Y$ is projectively $w^*$-rigged on the left over $M$, by the other-handed version of
Proposition \ref{new} there exists a pair $x = (x_i) \in {\rm Ball}(R^w_J(X)), y = (y_i) \in
 {\rm Ball}(C^w_J(Y))$, for a
cardinal $J$, such that $y_j \otimes x_i = [y_j,x_i] = \delta_{ij} p_j$ for an orthogonal projection $p_j \in M$, for
all $i,j \in J$, and $\sum_{j \in J} \, x_j(y_j) = 1$ weak* in $M$.
Define maps $\mu : N \to C^w_J(Y)$ and $\rho : C^w_J(Y) \to N$
by $\mu(n) =  [y_j n]$, and $\rho([z_j]) =  \sum_{j \in J} \, x_j(z_j)$, respectively.  We obtain $N \cong Q
C^w_J(Y)$, where $Q = \mu \circ \rho$.  We may identify $Q$ with
a diagonal matrix with projections $q_j$ as the diagonal entries.
If $L = (I-Q) C^w_J(Y)$,
then $N \oplus^c L \cong C^w_J(Y)$ as dual operator $N$-modules.  The Eilenberg swindle, as used in the proof of
\cite[Theorem 8.5.28]{BLM}, then yields $C^w_s(Y) \cong C^w_s(N)$, and $\Bbb{M}_s(Y)
\cong \Bbb{M}_s(N)$,
as dual operator $N$-modules, for some cardinal $s$.  By symmetry,
$\Bbb{M}_t(Y) \cong \Bbb{M}_t(M)$,  for some cardinal $t$ which we
can take to be equal to $s$.  Thus $M$ and $N$ are stably isomorphic as in \cite{BLM}.
 \end{proof}

\begin{proposition} \label{ne}    Not every $w^*$-rigged module is projectively $w^*$-rigged.
\end{proposition}

\begin{proof}  Suppose by way of contradiction that
every right $w^*$-rigged module $Y$ is a projectively $w^*$-rigged module.
Then every left $w^*$-rigged module is projectively $w^*$-rigged, by passing to the adjoint
(or conjugate) $\bar{Y}$.
Let $Y$ be the $M$-$N$-bimodule in Example (10) in \cite[Section 3]{BK1}.  This is an
example due to Eleftherakis, who showed that $Y$ does not
implement what he calls a $\Delta$-equivalence \cite{El,EP}.  We observed in
\cite{BK1} that $Y$ is a weak* Morita equivalence $M$-$N$-bimodule.
Hence it is both left and right $w^*$-rigged over
$M$ and $N$ respectively.  Thus  $Y$  is both left and right projectively
$w^*$-rigged.  By the last theorem, $M$ and $N$ are stably isomorphic,
hence $\Delta$-equivalent in Eleftherakis' sense.  This
is a contradiction.  \end{proof}

\section{Structure of isometries between bimodules}

A question of perennial interest is the structure of surjective
linear isometries between various algebras.  For $C^*$-algebras
the appropriate theorem is Kadison's noncommutative Banach-Stone
theorem.  For nonselfadjoint algebras, the most general results
on surjective
isometries are due to Arazy and Solel \cite{AS},
using deep techniques.  One gets a much simpler structure,
with a considerably easier proof, if one restricts
attention to surjective {\em complete} isometries
$T : A \to B$ between
unital operator algebras: the theorem here
is that $T$ is  the product of a unitary in $B \cap B^*$
and a surjective completely
isometric homomorphism from $A$ onto $B$ (see e.g.\ \cite[Theorem 4.5.13]{BLM}).  In this
section,
we are interested in generalizing such results to maps between bimodules.
Our spaces will be weak* Morita equivalence bimodules, that is bimodules that
are both left and right $w^*$-rigged modules, satisfying a natural compatibility
condition between the left and right actions \cite[Section 5 (3)]{BK2}
 (by considering 2 dimensional examples
it seems that
there is no characterization theorem in the case of general {\em one-sided}
$w^*$-rigged modules).
For $C^*$-modules (and hence $W^*$-modules), the structure of
surjective complete isometries follows immediately
from a theorem attributable to Hamana, Kirchberg, and Ruan, independently
\cite[Corollary 4.4.6]{BLM}.   The isometric case is due to Solel \cite{Sol},
who also shows that such
isometries lift to the `linking $C^*$-algebras'.     In \cite{DB3}, the first
author characterized the structure of
complete isometries between the strong Morita equivalence bimodules of \cite{BMP}.
Unfortunately  
the latter class of bimodules does not contain ours, and the proof techniques
used there fail in the `dual' setting, for example it employed the
noncommutative Shilov boundary, which has no
weak* topology variant to date.  We show here how this can be circumvented.
See the introduction for a discussion of the linking algebra (see
\cite[Section 4]{BK1} for more details if desired).

\begin{theorem} \label{nm}  Let $T : Y_1 \to Y_2$ be a surjective
linear complete isometry between weak* Morita equivalence bimodules in the
sense above.
Suppose that $Y_k$ is a weak* equivalence $M_k$-$N_k$-bimodule, for $k = 1,2$.  Then
there exist unique surjective completely
isometric homomorphisms $\theta : M_1 \to M_2$ and $\pi : N_1 \to N_2$
such that $T(ayb) = \theta(a) T(y) \pi(b)$ for all $a \in M_1, b \in N_1, y \in Y_1$.
If $T$ is weak* continuous then so are $\theta$ and $\pi$,
and moreover $T$ is the $1$-$2$-corner of a weak* continuous surjective completely
isometric homomorphism $\rho : {\mathcal L}^{\rm w}(Y_1) \to {\mathcal L}^{\rm w}(Y_2)$
between the weak linking algebras which maps corners to the matching corner.
\end{theorem}

\begin{proof}   We will use the machinery of `multipliers' of operator spaces (see
e.g.\ \cite{DB3}), which hitherto has
been the deepest tool in the theory of $w^*$-rigged modules \cite{BK2}.
From \cite[Theorem 2.3]{BK2}, and \cite[Theorem 3.6]{BK1}, we have that ${\mathcal M}_{\ell}(Y_k) = w^*CB(Y_k)_{N_k}
\cong M_k,$ and similarly ${\mathcal M}_r(Y_k) \cong N_k$, for $k = 1,2$.
By \cite[Proposition 4.5.12]{BLM},  the map ${\mathcal M}_{\ell}(Y_1) \to {\mathcal M}_{\ell}(Y_2) :u \mapsto T u T^{-1}$ is a
 completely
isometric isomomorphism.  It is also weak* continuous if $T$ is, using
\cite[Theorem 4.7.4]{BLM}.    Putting these maps together, we obtain
surjective completely
isometric homomorphisms $\theta : M_1 \to M_2$ and $\pi : N_1 \to N_2$, which
are weak* continuous if $T$ is.
We have $\theta(a)T(y) = T (a T^{-1}( T(y))) = T(ay)$ for $a \in M_1,
y \in Y$ as desired, and a similar
formula holds for $\pi$.  The uniqueness of $\theta, \pi$ is obvious.

Assuming $T$ weak* continuous, and with $X_k = \tilde{Y_k} = w^*CB(Y_k,N_k)_{N_k}$,
define $S : X_1 \to X_2$
by $S(x)(z) = \pi(x(T^{-1}(z)))$, for $x \in X_1,
z \in Y_2$.   It is routine to argue that $S$ is a complete isometry, since
$T$ and $\pi$ are.  Clearly $S(x)(T(y)) = \pi(x(y))$ for $x \in X_1, y \in Y_1$.
It is
simple algebra to check that $S(bxa) = \pi(b) S(x) (a)$, and
$[T(y),S(x)] = \theta([y,x])$.  To see the latter, for example, note that   for $y' \in Y$,
 $$[T(y),S(x)] T(y') = T(y) (S(x),T(y')) = T(y) \pi((x,y')) = T(y (x,y')) =
T([y,x] y'),$$
which is just $\theta([y,x]) T(y').$
Thus in a standard way, the maps under discussion become the four corners
of a surjective homomorphism $\rho$ between the weak linking algebras.   It is
easy to see that $\rho$ is weak* continuous, since each of its four corners
are.   That
$\rho$ is completely
isometric follows by pulling back the operator space
structure from ${\mathcal L}^{\rm w}(Y_2)$ via $\rho$, and
using the fact that there is a unique
operator space structure
on the weak linking algebra making it a dual operator
algebra, such that the
matrix norms on each of the four
corners is just the original norms of the
four spaces appearing in those corners
(see the third paragraph of \cite[Section 4]{BK1},
which appeals to the idea in \cite[p. \ 50--51]{BLM}).
\end{proof}

\medskip

{\bf Acknowledgements:} \  Much of the
content of this paper was found in discussions
between the two authors some months after \cite{BK2} was written, 
and advertized informally then.  The work was
not published at that time since there was another avenue that the authors were exploring
together,
namely a nonselfadjoint variant of some of the theory presented in \cite{Kraus}.
Although many basic aspects of that theory do carry over, some important parts
of that program failed to generalize, and therefore we have decided not to publish
this.

\end{document}